\documentclass[fleqn,11pt]{article}
\usepackage{geometry}

\usepackage{amsmath,amssymb,commath,mathtools,cancel}

\newcommand{\bigpa}[1]{\bigl(#1\bigr)}
\newcommand{\Bigpa}[1]{\Bigl(#1\Bigr)}
\newcommand{\biggpa}[1]{\biggl(#1\biggr)}
\newcommand{\brc}[1]{\{#1\}}
\newcommand{\Int}{\mathbb{Z}}

\usepackage[amsmath,thmmarks,standard]{ntheorem}

\usepackage{cleveref}

\usepackage{tikz}
\usetikzlibrary{arrows.meta}
\usetikzlibrary{shapes.misc}
\usepackage{listofitems}
\usepackage[export]{adjustbox}

\usepackage{verbatim}
\usepackage{url}

\begin{document}

\title{The Chv\'atal--Sankoff problem: 
Understanding random string comparison through stochastic processes}

\author{Alexander Tiskin%
\footnote{This work was supported by the Russian Science Foundation under grant no.\ 22-21-00669, 
\protect\url{https://www.rscf.ru/en/project/22-21-00669}}}

\date{Department of Mathematics and Computer Science\\ St.~Petersburg State University}

\maketitle

\begin{abstract}

Given two equally long, uniformly random binary strings, 
the expected length of their longest common subsequence (LCS)
is asymptotically proportional to the strings' length.
Finding the proportionality coefficient $\gamma$,
i.e.\ the limit of the normalised LCS length for two random binary strings of length $n \to \infty$, 
is a very natural problem, 
first posed by Chv\'atal and Sankoff in 1975, and as yet unresolved.
This problem has relevance to diverse fields ranging from combinatorics and algorithm analysis
to coding theory and computational biology.
Using methods of statistical mechanics,
as well as some existing results on the combinatorial structure of LCS,
we link constant $\gamma$ to the parameters of a certain stochastic particle process,
which we use to obtain a new estimate for $\gamma$%
%
\footnote{
In the preprint version of this paper, 
certain claims were made regarding the nature of this process and the constant $\gamma$,
which subsequently turned out to be incorrect.
The erroneous parts of the preprint are omitted from the paper,
while keeping the partial result on an estimate for $\gamma$ supported by our construction.}.
%

\end{abstract}
 
\allowdisplaybreaks

\newcommand{\colcond}{red!50}
\newcommand{\textcond}[1]{\textcolor{\colcond}{#1}}
\newcommand{\colforced}{blue!50}
\newcommand{\textforced}[1]{\textcolor{\colforced}{#1}}

\newcommand{\sitev}[2]{
\ifthenelse{\equal{#1}{}}{}{
\draw #2 +(0,-0.5) -- +(0,0.5);}
\site{#1}{#2}}

\newcommand{\siteh}[2]{
\ifthenelse{\equal{#1}{}}{}{
\draw #2 +(-0.5,0) -- +(0.5,0);}
\site{#1}{#2}}

\newcommand{\site}[2]{
\ifthenelse{\equal{#1}{0}}{\node at#2[circle,draw,fill=white]{};}{   
\ifthenelse{\equal{#1}{1}}{\node at#2[circle,draw,fill]{};}{         
\ifthenelse{\equal{#1}{x}}{                                          
  \node at#2[circle,draw,fill=white]{};
  \node at#2[circle,draw,fill,scale=0.5]{};}{
\ifthenelse{\equal{#1}{o}}{\node at#2[circle,draw=\colcond,fill=white]{};}{   
\ifthenelse{\equal{#1}{i}}{\node at#2[circle,draw=\colcond,fill=\colcond]{};}{  
\ifthenelse{\equal{#1}{O}}{\node at#2[circle,draw=\colforced,fill=white]{};}{  
\ifthenelse{\equal{#1}{I}}{\node at#2[circle,draw=\colforced,fill=\colforced]{};}{
\ifthenelse{\equal{#1}{B}}{
  \node at#2[fill=white,inner sep=1pt,outer sep=0pt]{\textforced{$\scriptstyle b$}};}{
\ifthenelse{\equal{#1}{C}}{
  \node at#2[fill=white,inner sep=1pt,outer sep=0pt]{\textforced{$\scriptstyle c$}};}{
\ifthenelse{\equal{#1}{.}}{}{
\ifthenelse{\equal{#1}{}}{}{
\node at#2[fill=white,inner sep=1pt,outer sep=0pt]{$\scriptstyle #1$};
}}}}}}}}}}}}

\newcommand{\slot}[2]{
\ifthenelse{\equal{#1}{1}}{\draw[very thick] #2 +(-0.25,0.25) -- +(0.25,-0.25);}{        
\ifthenelse{\equal{#1}{0}}{\draw[very thick] #2 +(-0.25,-0.25) -- +(0.25,0.25);}{        
\ifthenelse{\equal{#1}{*}}{
  \draw[very thick] #2 +(-0.25,0.25) -- +(0.25,-0.25) +(-0.25,-0.25) -- +(0.25,0.25);}{  
\ifthenelse{\equal{#1}{I}}{\draw[very thick,\colforced] #2 +(-0.25,0.25) -- +(0.25,-0.25);}{
\ifthenelse{\equal{#1}{O}}{\draw[very thick,\colforced] #2 +(-0.25,-0.25) -- +(0.25,0.25);}{
\ifthenelse{\equal{#1}{i}}{\draw[very thick,\colcond] #2 +(-0.25,0.25) -- +(0.25,-0.25);}{ 
\ifthenelse{\equal{#1}{o}}{\draw[very thick,\colcond] #2 +(-0.25,-0.25) -- +(0.25,0.25);}{ 
\ifthenelse{\equal{#1}{x}}{
  \draw[very thick,\colcond] #2 +(-0.25,0.25) -- +(0.25,-0.25) +(-0.25,-0.25) -- +(0.25,0.25);}{
\ifthenelse{\equal{#1}{}}{}{
\node at#2[fill=white,inner sep=1pt,outer sep=0pt,opacity=0,text opacity=1]{$\scriptstyle #1$};
}}}}}}}}}}

\newcommand{\animal}[5][]{\setsepchar{:}
\readlist\slots{#2} \readlist\sitesa{#3} \readlist\sitesb{#4} \readlist\sitesc{#5}
\mathord{\begin{tikzpicture}[baseline=(base),x=1em,y=-1em,
  every circle node/.style={draw,thick,inner sep=0pt,outer sep=0pt,minimum size=1.5mm}]
\siteh{\sitesa[1]}{(-0.5,1)} \sitev{\sitesa[2]}{(0,0.5)} 
\siteh{\sitesa[3]}{(0.5,0)}  \sitev{\sitesa[4]}{(1,-0.5)}
\sitev{\sitesb[1]}{(0,1.5)}  \siteh{\sitesb[2]}{(0.5,1)}
\sitev{\sitesb[3]}{(1,0.5)}  \siteh{\sitesb[4]}{(1.5,0)}
\siteh{\sitesc[1]}{(0.5,2)}  \sitev{\sitesc[2]}{(1,1.5)}
\siteh{\sitesc[3]}{(1.5,1)}  \sitev{\sitesc[4]}{(2,0.5)}
\slot{\slots[1]}{(0.5,0.5)} \slot{\slots[2]}{(0.5,1.5)} 
\slot{\slots[3]}{(1.5,0.5)} \slot{\slots[4]}{(1.5,1.5)}
%
\ifthenelse{\equal{#1}{x}}{
  \draw[overlay,ultra thick,black!50] (1,1) +(-0.5,-1.5) -- +(0.5,1.5) +(-1.5,0.5) -- +(1.5,-0.5);}{}
\path (current bounding box.center) coordinate[below=0.5ex](base);
\end{tikzpicture}}}

\newcommand{\zvh}[2][]{%
\setsepchar{:} \readlist\sites{#2}
\animal{#1:::}{:\sites[1]:\sites[2]:}{:::}{:::}}

\newcommand{\zhv}[2][]{%
\setsepchar{:} \readlist\sites{#2}
\animal{#1:::}{:::}{:\sites[1]:\sites[2]:}{:::}}

\newcommand{\zhvh}[2][]{%
\setsepchar{:} \readlist\sites{#2}
\animal{#1:::}{\sites[1]:\sites[2]:\sites[3]:}{:::}{:::}}

\newcommand{\zvhv}[2][]{%
\setsepchar{:} \readlist\sites{#2}
\animal{#1:::}{:\sites[1]:\sites[2]:\sites[3]}{:::}{:::}}

\newcommand{\zhvhv}[2][]{%
\setsepchar{:} \readlist\sites{#2}
\animal{#1:::}{\sites[1]:\sites[2]:\sites[3]:\sites[4]}{:::}{:::}}

\newcommand{\zvhvh}[2][]{%
\setsepchar{:} \readlist\sites{#2}
\animal{#1:::}{:::}{\sites[1]:\sites[2]:\sites[3]:\sites[4]}{:::}}

\newcommand{\cell}[3][]{%
\setsepchar{:} \readlist\sitesaa{#2} \readlist\sitesbb{#3}
\animal{#1:::}{:\sitesaa[1]:\sitesaa[2]:}{:\sitesbb[1]:\sitesbb[2]:}{:::}}

\newcommand{\crab}[3][::::]{\setsepchar{:}
\readlist\slots{#1} \readlist\sitesa{#2} \readlist\sitesb{#3}
\mathord{\begin{tikzpicture}[baseline=(base),x=1em,y=-1em,
  every circle node/.style={draw,thick,inner sep=0pt,outer sep=0pt,minimum size=1.5mm}]
\sitev{\sitesa[1]}{(-1,1.5)} \siteh{\sitesa[2]}{(-0.5,1)} \sitev{\sitesa[3]}{(0,0.5)}
\siteh{\sitesa[4]}{(0.5,0)}  \sitev{\sitesa[5]}{(1,-0.5)} \siteh{\sitesa[6]}{(1.5,-1)}
\siteh{\sitesb[1]}{(-0.5,2)} \sitev{\sitesb[2]}{(0,1.5)}  \siteh{\sitesb[3]}{(0.5,1)}
\sitev{\sitesb[4]}{(1,0.5)}  \siteh{\sitesb[5]}{(1.5,0)}  \sitev{\sitesb[6]}{(2,-0.5)}
\slot{\slots[1]}{(-0.5,1.5)} \slot{\slots[2]}{(0.5,0.5)} \slot{\slots[3]}{(1.5,-0.5)} 
\slot{\slots[4]}{(0.5,1.5)} \slot{\slots[5]}{(1.5,0.5)}
\path (current bounding box.center) coordinate[below=0.5ex](base);
\end{tikzpicture}}}

\renewcommand{\u}{u} \newcommand{\U}{\bar u}
\newcommand{\V}{\bar v}
\newcommand{\w}{w} \newcommand{\W}{\bar w}
\newcommand{\y}{y} \newcommand{\Y}{\bar y}
\newcommand{\z}{z} \newcommand{\Z}{\bar z}

\newcommand{\sub}[1]{
\foreach \k in {#1}{
\ifthenelse{\equal{\k}{0}}{\circ}{
\ifthenelse{\equal{\k}{1}}{\bullet}{\k}
}}}

\newcommand{\pp}[1]{p_{\sub{#1}}}
\newcommand{\PP}[1]{\bar p_{\sub{#1}}}

\def\p{p}
\def\P{\bar p}
\def\q{q}
\def\Q{\bar q}
\def\r{r}
\def\R{\bar r}

\newcommand{\s}{s}
\renewcommand{\S}{\bar s}
\newcommand{\g}{g}
\newcommand{\G}{\bar g}
\newcommand{\h}{h}
\renewcommand{\H}{\bar h}

\newcommand{\cellshape}[1]{
\mathord{\begin{tikzpicture}[baseline=(base),x=0.5em,y=-0.5em]
\foreach \i/\j in {#1}
  \draw (\i,\j) rectangle +(1,1);
\path (current bounding box.center) coordinate[below=0.5ex](base);
\end{tikzpicture}}}

\newcommand{\cellthree}{\cellshape{0/0,0/1,1/0}}
\newcommand{\cellfour}{\cellshape{0/0,0/1,1/0,1/1}}

\newcommand{\eqdef}{\overset{\mathsf{def}}{=}}
\newcommand{\eqjda}{\overset{\mathsf{def}}{=}}
\newcommand{\eqs}{\overset{\mathsf{s}}{=}}
\newcommand{\eqr}{\overset{\mathsf{r}}{=}}
\newcommand{\eqsr}{\overset{\mathsf{sr}}{=}}

\section{Introduction}

The \emph{longest common subsequence (LCS)} for a pair of strings $a$, $b$
is the longest string that is a (not necessarily consecutive) subsequence of both $a$ and $b$.
Given a pair of strings as input, the \emph{LCS problem}
asks for the length of their LCS (finding the actual characters of the LCS is not required).
The LCS problem is a fundamental problem for both theoretical and applied computer science,
and for computational molecular biology; it is also a popular programming exercise.

This paper is concerned with the combinatorics of the LCS problem.
Let strings $a$, $b$ be of length $n$, uniformly random over the binary alphabet.
Chv\'{a}tal and Sankoff \cite{Chvatal_Sankoff:75} (see also \cite[Chapter 1]{Steele:97}) 
have shown that the expected LCS length of $a$, $b$ is asymptotically proportional to $n$.
The \emph{Chv\'{a}tal--Sankoff problem} asks for the proportionality coefficient $\gamma$,
i.e.\ the limit of the normalised expected LCS length $\frac{\mathbb{E} L_n}{n}$ 
as $n \to \infty$, where the random variable $L_n$ is defined as the LCS length for strings of length $n$.
Alexander \cite{Alexander:94} has shown that 
$0 \leq \gamma - \frac{\mathbb{E} L_n}{n} \leq O\bigpa{\bigpa{\frac{\log n}{n}}^{1/2}}$.

The Chv\'{a}tal--Sankoff problem has relevance to diverse fields 
ranging from combinatorics and algorithm analysis 
to coding theory (see e.g.\ Bukh et al.\ \cite{Bukh+:17})
and computational biology (see e.g.\ Pevzner and Waterman \cite{Pevzner_Waterman:95}).
For such a natural and simply posed problem, it seems to be surprisingly elusive:
neither an exact value nor any closed-form expression for $\gamma$ are known,
and the existing lower and upper numerical bounds on $\gamma$ are wide apart.

\paragraph{Acknowledgements}
I thank Gianfranco Bilardi, Chris Cox, Vassily Duzhin, Maria Fedorkina, Sergei Nechaev, 
Georgiy Shulga, Nikolai Vassiliev, and Anatoly Vershik for fruitful discussions.
I thank my colleagues and students at the Department of Mathematics and Computer Science 
of St.~Petersburg University for the stimulating atmosphere.

\section{Related work}

\paragraph{LCS combinatorics}
An important combinatorial feature of the LCS problem, 
also relevant to its computational aspect, 
is the problem's close connection with transposition networks and the Hecke monoid 
(also called the seaweed monoid or the sticky braid monoid).
This connection has been explored over decades from different angles 
and using greatly varying terminology.
In the rest of this paper, we will describe this connection in more detail,
and will use it as the first step on our path to the Chv\'{a}tal--Sankoff problem.

While the computational aspect of the LCS problem is outside the scope of this paper,
it should be mentioned that the problem's computational complexity,
along with that of the closely related edit distance and sequence alignment problems,
has been thoroughly studied and is well-understood.
Seminal work on LCS algorithms and lower bounds includes
e.g.\ \cite{Wagner_Fischer:74,Masek_Paterson:80,Abboud_:15,Bringmann_Kuenemann:18}.

\paragraph{Random LCS on permutation strings}
Apart from binary strings,
a question analogous to the Chv\'{a}tal--Sankoff problem can be asked 
about pairs of uniformly random permutations of the alphabet $\brc{1,\ldots,n}$.
The LCS problem on such permutation strings
is equivalent to finding the longest increasing subsequence (LIS) 
of a single permutation of length $n$.
The LCS (respectively, LIS) length in this case turns out to be asymptotically proportional to $\sqrt{n}$.
The proportionality constant was found to be exactly $2$ in the classical works 
of Vershik and Kerov \cite{Vershik_Kerov:77} and Logan and Shepp \cite{Logan_Shepp:77}
(see also \cite{Romik:14}), as part of a solution for the more general problem 
asking for the limit shape of a random Young diagram sampled from the Plancherel distribution.

\paragraph{Bounds and estimates for $\gamma$}

Chv\'{a}tal and Sankoff  \cite{Chvatal_Sankoff:75} gave the first analysis of the problem,
and proved the existence of the limit $\gamma$.
Properties of the convergence of the normalised LCS length to this limit
were studied since then by numerous researchers.
\Cref{t-bounds} lists some results on specific lower and upper bounds,
as well as experimental numerical estimates of $\gamma$.

\begin{table}\centering
\begin{tabular}{|l|l|l|l|}\hline
Reference & $\gamma > {}$ & $\gamma$ & $\gamma < {}$ \\ \hline
Chv\'{a}tal and Sankoff  \cite{Chvatal_Sankoff:75} &
$0.697844$ & $\approx 0.8082$ & $0.866595$\\
Deken \cite{Deken:79} & 
$0.7615$ & & $0.8575$\\
Steele \cite{Steele:86} (conjecture attr.\ to Arratia) &
& $\stackrel{?}{=} 2(\sqrt2-1) \approx 0.8284$ &\\
Dan\v{c}\'ik \cite{Dancik:94}; Paterson and Dan\v{c}\'ik \cite{Paterson_Dancik:94} &
$0.77391$ & $\approx 0.812$ & $0.83763$ \\
Baeza-Yates et al.\ \cite{BaezaYates+:99} &
& $\approx 0.8118$ &\\
Boutet de Monvel \cite{Boutet:99} & 
& $\approx 0.812282$ &\\
Bundshuh \cite{Bundschuh:01} & 
& $\approx 0.812653$ &\\
Lueker \cite{Lueker:09} &
$0.788071$ & & $0.826280$\\
Bukh and Cox \cite{Bukh_Cox:22} &
& $\approx 0.8122$ &\\ \hline
%
%
\end{tabular}
\caption{\label{t-bounds}Bounds and estimates on $\gamma$}
\end{table}

The best currently known analytic bounds on $\gamma$ are due to Lueker \cite{Lueker:09}.
Despite the ingenious methods of obtaining these bounds and numerous related results,
the gap between the upper and the lower bounds remains quite wide: 
in particular, not a single digit of $\gamma$ after decimal point is known exactly.

\paragraph{Stochastic evolution models}
Due to the combinatorial properties of the LCS problem that will be presented in the next section, 
the Chv\'{a}tal--Sankoff problem turns out to be closely related to the theory of stochastic evolution models,
which is a vast and actively developing field of study.
Particularly relevant areas within this field include particle processes,
random Young diagrams, stochastic cellular automata.
Asymptotic properties of such models are studied with the help of partial differential equations (PDEs),
which describe a model's evolution at the macroscopic level.
In the rest of this paper, we will describe these connections in more detail.

\section{Combinatorics of the LCS problem}

\paragraph{LCS grid}
Let strings $a$, $b$ be of length $m$, $n$ respectively.
The \emph{LCS grid} defined by $a$, $b$ is a directed graph 
on an $(m+1) \times (n+1)$ grid of nodes;
we visualise the nodes as being indexed top-to-bottom and left-to-right.
Every pair of horizontally or vertically adjacent nodes
are connected by an edge, directed rightwards (respectively, downwards).
A pair of diagonally adjacent nodes $(i,j)$, $(i+1,j+1)$, $0 \leq i < m$, $0 \leq j < n$,
are connected by an edge whenever $a_i = b_j$ (the two characters \emph{match}); 
this edge is directed towards below-right.
The LCS grid can also be viewed as an $m \times n$ grid of cells,
each formed by a quadruple of adjacent nodes and their four connecting horizontal and vertical edges.
The cell is called \emph{match cell}, if the two corresponding characters match
(and therefore the cell contains a diagonal edge), otherwise a \emph{mismatch cell}.
The LCS problem is equivalent to asking for the length of a path in the LCS grid
from the top-left node $(0,0)$ to the bottom-right node $(m,n)$,
that maximises the number of diagonal edges along the path.

\begin{example}
\Cref{f-example} (left) shows the LCS grid for a pair of binary strings.
The horizontal and vertical edges are shown in light-blue, and the diagonal edges in solid red.
The left-to-right, top-to-bottom direction of the edges is left implicit.
\end{example}

\begin{figure}\centering
\includegraphics[valign=c]{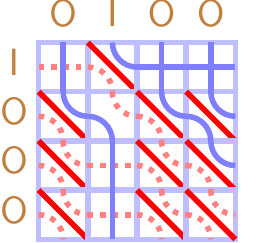}\qquad
\includegraphics[valign=c]{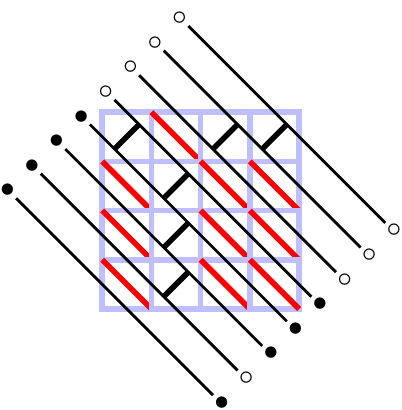}\qquad
\includegraphics[valign=c]{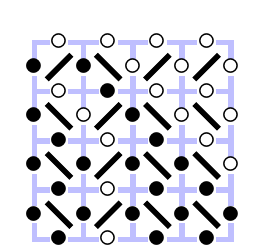}
\caption{\label{f-example}LCS grid with a sticky braid (left), 
transposition network (centre), particle evolution model (right) 
for strings $a = \textsf{``IOOO''}$, $b = \textsf{``OIOO''}$}
\end{figure}

\paragraph{Sticky braids}
The combinatorial structure of the LCS problem
is described algebraically by the \emph{Hecke monoid} (also known as the \emph{sticky braid monoid}),
which is defined similarly to the classical braid group,
but with element inversion replaced by the idempotence relation on the monoid's generators.
Given an LCS grid, strands of the corresponding sticky braid are formed by paths in the dual graph,
i.e.\ the plane graph whose nodes are the faces of the LCS grid, 
and the edges go across the edges of the LCS grid.
Multiplication of sticky braids in the Hecke monoid (also known as \emph{Demazure multiplication})
describes precisely how LCS lengths of input strings and their substrings behave under string concatenation.

\begin{example}
\Cref{f-example} (left) shows a sticky braid embedded into the LCS grid of the previous example.
The braid's strands are shown in darker blue and dotted red.
\end{example}

The connection outlined above between the LCS problem and the Hecke monoid 
has been rediscovered many times in different forms.
In particular, it underlies implicitly the algorithms for various string comparison problems
by Schmidt \cite{Schmidt:98}, Crochemore et al.\ \cite{Crochemore+:01_IPL,Crochemore+:03_SIAM}, 
Alves et al.\ \cite{Alves+:08}, Hyyr\"o \cite{Hyyro:17},
and was made explicit by Tiskin \cite{Tiskin:08_JDA,Tiskin:08_MCS,Tiskin:15_Algorithmica}.
More recently, new algorithmic applications of this connection
were found by Sakai \cite{Sakai:11,Sakai:19}, 
Tiskin \cite{Tiskin:09_CPM,Tiskin:11_CSR,Tiskin:19_WABI,Tiskin:2020_SPAA},
Gawrychowski et al.\ \cite{Gawrychowski:12}, Hermelin et al.\ \cite{Hermelin+:13}, 
Matarazzo et al.\ \cite{Matarazzo+:14}, 
Charalampopoulos et al.\ \cite{Charalampopoulos+:20,Charalampopoulos+:21}.

\paragraph{Transposition networks}
Another convenient tool for exposing the combinatorial structure of the LCS problem
comes in the form of \emph{transposition networks}.
These are a special case of comparison networks, which are a classical type of computational circuits 
studied by Batcher \cite{Batcher:68}, Knuth \cite{Knuth:98_3} and many others.
In a comparison network, input values travel on an array of parallel wires;
any prescribed pair of values can be sorted by a \emph{comparator} connecting their respective wires.
In a transposition network, an additional restriction is imposed 
that only adjacent pairs of wires can be connected by a comparator.

Given a pair of strings $a$, $b$ of lengths $m$, $n$ respectively, 
their LCS grid can be overlayed by a transposition network on $m+n$ wires,
extending diagonally from above-left to below-right and passing through the midpoints of the grid's edges.
These intersection points of the network's wires and the grid's edges will be called \emph{sites};
we will distinguish horizontal and vertical edge sites.
A wire passes through an alternating sequence of horizontal and vertical edge sites;
the \emph{value} of a given site is the value carried through it by the wire.
A cell is crossed by two wires: 
one connecting its left and bottom boundary edges, the other its top and right boundary edges.
The two sites at the cell's left and top boundary edges are its \emph{entry sites}, 
and the two sites at its right and bottom boundary edges are its \emph{exit sites}.
The network's comparators are specified as follows:
a mismatch cell always contains a comparator between the two wires that cross it,
while a match cell never contains a comparator.
A cell can therefore be of one of two \emph{types}:
``match'' (denoted `$\zvh[0]{:}$'), containing a diagonal grid edge,
and ``mismatch'' (denoted `$\zvh[1]{:}$'), containing a network's comparator;
the notation indicates the direction of the diagonal edge and of the comparator, respectively.
Occasionally, we identify cell type `$\zvh[0]{:}$' with value zero, 
and cell type `$\zvh[1]{:}$' with value one.

\begin{example}
\Cref{f-example} (centre) shows the LCS grid of the previous example,
overlaid with its respective transposition network.
\end{example}

Given an input of $m+n$ distinct values sorted in reverse order,
the set of values' trajectories through such a transposition network 
forms a sticky braid corresponding to the comparison of strings $a$, $b$;
each particular value traces a strand in this braid.
The network's output permutation provides detailed information 
about LCS lengths between various substrings of $a$, $b$.
For our purposes, the above construction can be simplified as follows:
instead of all distinct values, let the transposition network's input 
consist of $m$ ones, followed by $n$ zeros; 
note that such an input array is still sorted in reverse.
In this context, 
value zero will be called a \emph{hole} (denoted `$\circ$'),
and value one a \emph{particle} (denoted `$\bullet$').
This is done not only to distinguish the (binary) values in the network
from (also binary) string characters and (again binary) cell types,
but also to reflect in our terminology the important connection with particle interaction models,
that we will develop further in the remainder of this paper.

An assignment of values/types to a subset of sites/cells of a transposition network
will be called a \emph{configuration}.
In particular, the input configuration
formed by $m$ particles entering the LCS grid at its left boundary,
and $n$ holes entering at the top boundary, 
will be called the \emph{step initial condition}. 

\begin{example}
The transposition network in \Cref{f-example} (centre) is shown 
with the step initial condition input sequence at the top-left,
and the corresponding output sequence of particles and holes at the bottom-right.
\end{example}

The LCS length of strings $a$, $b$ is particularly easy to obtain 
from the transposition network with step initial condition:
it is equal to the number of particles among the network's $n$ outputs exiting the grid at the bottom
(equivalently, the number of holes among its $m$ outputs exiting the grid at the right).
This observation underlies implicitly the bit-parallel LCS algorithms 
of Crochemore et al.\ \cite{Crochemore+:01_IPL} and Hyyr\"o \cite{Hyyro:17},
and was made explicitly e.g.\ by Majumdar and Nechaev \cite{Majumdar_Nechaev:05}
and by Krusche and Tiskin \cite{Krusche_Tiskin:09}.
Let $a$, $b$ be of equal length $m=n$; in this case, the LCS grid has the shape of a square,
and the LCS length is equal to the number of particles (equivalently, the number of holes)
that have never crossed the grid's main diagonal.

\begin{example}
In the previous example, there are three particles among the $n=4$ outputs at the grid's bottom;
the LCS length for strings $a$, $b$ is also $3$.
In the course of the evolution of the transposition network,
$4-3=1$ particle has crossed the main diagonal from left to right;
accordingly, one hole has done so from top to bottom.
\end{example}

\section{Model $\mathit{CS}$}

The combinatorial properties of the LCS problem
allow us to reformulate the Chv\'atal--Sankoff problem
in the language of stochastic particle interaction models.
By a \emph{network evolution model}, we will understand the evolution of site values 
from a given input configuration in an infinitely wide transposition network, 
under a certain probabilistic rule that determines the type of each of the network's cells.

\paragraph{Cell dependencies}
Let $a$, $b$ now be infinite strings, where all characters are independent uniform binary random variables.
We define \emph{model $\mathit{CS}$} (the Chv\'atal--Sankoff model) as a network evolution model
where cell types are determined by character matches and mismatches between strings $a$, $b$,
as described in the previous section.
\begin{proposition}
\label{p-three-ind}
In model $\mathit{CS}$, the types of any three distinct cells are mutually independent.
The types of any three distinct cells within a $\cellfour$-shape determine uniquely the type of the fourth cell.
\end{proposition}
\begin{proof}
The first statement is straightforward by the independence and uniformity 
of character distribution in strings $a$, $b$.
The second statement is also straightforward, since the sum of the four cells' types must be even.
\end{proof}
In particular, the types of any three cells 
adjacent in a $\cellthree$-shape are mutually independent;
we shall call this property \emph{$\cellthree$-independence}.
Note that $\cellthree$-independence relies crucially the uniform distribution of string characters,
and would not hold for a non-uniform character distribution, even if it were independent and identical.

\paragraph{Evolution}
Let strings $a$, $b$ be indexed by $i$, $j$ respectively.
The state of model $\mathit{CS}$ can be thought of as evolving in several different ways ---
vertically, horizontally or diagonally, 
with the discrete time dimension indexed by $i$, $j$ and $\frac{i+j}{2}$, respectively.
We will focus mainly on the diagonal evolution, due to its symmetry and locality properties.
The model's state under such evolution 
corresponds to an anti-diagonal doubly-infinite sequence of particle-hole values,
alternating between horizontal and vertical edge sites.
Let us index the transposition network's wires entering the grid through its top boundary 
with nonnegative integers $0, 1, 2, \ldots$,
and the wires entering the grid through its left boundary with negative integers $-1, -2, -3, \ldots$;
the count in both cases starts from the top-left cell.
A time step under diagonal evolution then consists of two half-steps:
the first involves comparators operating on pairs of adjacent sites with an odd and an even index (in that order),
the second on pairs with an even and an odd index (in that order).

As discussed in the previous section, the behaviour of model $\mathit{CS}$
reflects the LCS combinatorics of its underlying string pair $a$, $b$.
\begin{proposition}
\label{p-cs-shape}
Let $0 \leq k \leq 2n$.
Consider the prefixes of infinite strings $a$, $b$ of length $k$, $2n-k$ respectively,
and let $l$ be the LCS length of these prefixes.
Under diagonal evolution of model $\mathit{CS}$ from step initial condition after $n$ time steps,
there are $k-l$ particles at sites with indices $2n-2k$ or greater.
\end{proposition}
\begin{proof}
Well-known from the combinatorial properties of LCS; see e.g.\ \cite{Majumdar_Nechaev:05,Krusche_Tiskin:09}.
\end{proof}

\begin{example}
\Cref{f-example} (right) shows the evolution of model $\mathit{CS}$ from step initial condition
on strings $a$, $b$ of the previous examples.
Wires with negative (respectively, nonnegative) indices are those below (respectively, above)
the network's main diagonal.
Let $n = k = 4$.
The LCS length of the input strings, 
regarded as prefixes of length $k = 2n-k = 4$ of a pair of infinite strings, is $l=3$; 
as before, we note that after $n = 4$ time steps, exactly $n-l = 4-3 = 1$ particle 
has crossed over the main diagonal to wires with nonnegative indices.
\end{example}

\paragraph{Duality}

The definition of model $\mathit{CS}$
is symmetric with respect to the reflection of the network about its main diagonal.
A pair of configurations will be called \emph{dual}, if one of them is obtained from the other
by a reflection about an above-left to below-right axis 
(exchanging the directions towards below-left and above-right),
with simultaneous exchange of sites' values between particles and holes.
In particular, the step initial condition is a self-dual configuration.

In the remainder of this paper, we will consider model $\mathit{CS}$ with step initial condition.
Our analysis will concentrate on the model's behaviour in a small neighbourhood of the main diagonal, 
where the particle and hole densities should be asymptotically equal by symmetry.
Duality will help to simplify the exposition,
since in such a setting, a pair of dual configurations will have equal probabilities.

\section{Special notation}

\paragraph{Configuration probabilities}

We consider configurations of a network evolution model as random events.
The probability of an event will be denoted by its graphical representation.
Thus, $\zvh{1:} = 1 - \zvh{0:}$ represents the probability of a given vertical edge site
holding a particle, as opposed to a hole,
and $\zvh[1]{:} = 1 - \zvh[0]{:}$ represents the probability of a given cell
being of type ``mismatch'', as opposed to ``match''.

We extend this notation to represent conditional probabilities as follows.
We juxtapose the conditioning event and the conditioned event in the same picture;
the elements of the conditioning event will be highlighted in red,
while the elements of the conditioned event will be shown in the ordinary black.
For example, the probability of a given cell being of type ``mismatch'',
conditioned on the cell's left (respectively top) entry value 
being a particle (respectively, a hole), will be denoted by
$\zvh[1]{i:o} = \zvh[1]{1:0} / \zvh{1:0}$.

Some events that we consider may be forced by other events:
a forced event, conditioned on the forcing event, occurs with certainty.
We juxtapose the forcing event and the forced event in the same picture;
the elements of the forced event will be highlighted in blue,
while the elements of the forcing event will be shown in either black or red, as appropriate.
In the previous example, the cell's exit values are forced:
$\cell[1]{i:o}{:} = \cell[1]{i:o}{O:I}$.
Showing forced sub-events is a notational decoration that can formally be omitted;
however, it is meant to serve as an intuition aid,
especially so when some non-forced sub-event becomes forced in a chain of equalities.
For example, we have $\cell{:}{0:1} = \cell{0:1}{O:I} + \cell[1]{1:0}{O:I}$.

\paragraph{Annotated equalities}
Standard annotated equality $A \eqdef B$ (``$A$ is defined as $B$'') will be used to introduce new notation.
Additionally, we will use some other annotations on the equality sign, as an aid to the reader.
Notation $A \eqr B$ (``$A$ and $B$ are obtained from each other 
by reversal with an exchange of particles and holes'')
will indicate that the equality holds by the duality property of network configurations.

\paragraph{Other notation}
For brevity, we will denote $\bar z = 1-z$ for any $z$, $0 \leq z \leq 1$.
Strings in the alphabet $\brc{\circ,\bullet}$ will sometimes be treated as binary numbers;
for brevity, we will convert such numbers to decimal where appropriate.
We let $a,b \in \brc{\circ,\bullet}$ for the remainder of this paper.

\section{Scaling limits}

Informally, the \emph{scaling limit} of a particle evolution model 
is the continuous limit of the distribution of particle densities at the model's sites,
as both time and space are simultaneously scaled down at appropriate rates,
so that the magnitude of both time and space units tends to zero.
A general introduction to the theory of scaling limits 
is given e.g.\ by Kriecherbauer and Krug \cite{Kriecherbauer_Krug:10}.

\paragraph{Scalar conservation laws}

Partial differential equations (PDEs) are an indispensable tool in studying
the asymptotic behaviour of particle evolution models.
Using PDEs, one can relate the global behaviour of the model,
such as its non-stationary evolution from a given initial condition, with its local behaviour,
such as its stationary state in a small space-time region.
A classical example of such a relationship is the asymptotic behaviour 
of the continuous-time totally asymmetric simple exclusion process (TASEP) with step initial condition,
which was shown to be governed by the inviscid Burgers' equation by Rost \cite{Rost:81} 
(see also \cite{Kriecherbauer_Krug:10,Romik:14,Ferrari:18}).

In general, the scaling limit of a conservative particle model with one spatial dimension
can be associated with a \emph{scalar conservation law} (see e.g.\ \cite{Kriecherbauer_Krug:10}), 
which is a PDE of the form
\newcommand{\dd}[1]{\tfrac{\partial}{\partial #1}}
\begin{gather*}
\dd{t} y + \dd{x} f(y) = 0
\end{gather*}
where $y = y(t,x)$ is the \emph{density} function of time $t$ and the spatial dimension $x$,
representing the conserved quantity (typically, the mass of some fluid),
and $f = f(y)$ is a strictly concave smooth function of density $y$ called the (rightward) \emph{flux}.
We are particularly interested in the \emph{step} initial condition:
\begin{gather*}
y(0,x) = 
\begin{cases}
1 & x < 0\\
0 & x > 0
\end{cases}
\end{gather*}
In the language of PDEs, the step initial condition 
is a special case of the Riemann problem for a scalar conservation law.
The discontinuity of $y$ at $x=t=0$ is known as \emph{shock}.
This initial shock dissipates over time in a \emph{rarefaction wave},
governed by the equation's solution (see e.g.\ \cite{Kriecherbauer_Krug:10,Salsa:16}) 
\begin{gather*}
y(t,x) = 
\begin{cases}
(f')^{-1} (x/t) & f'(1) t \leq x \leq f'(0) t\\
y(0,x) & \text{otherwise}
\end{cases}
\end{gather*}
where $f'$ is the derivative of $f$, and superscript $-1$ denotes its functional inverse.

Since the solution scales linearly with $t$, 
it is sufficient for the analysis to consider a single time moment $t > 0$;
a natural choice is $t=1$.
Let $y(x) = y(1,x)$.
We impose further constraints $0 \leq y \leq 1$, $f(0) = f(1) = 0$,
which are natural for the interpretation of $y$ as a fluid's density.
The maximum flux $\tilde f$ is determined by $f'(y) = 0$,
and is therefore attained at density $\tilde y = (f')^{-1}(0) = y(0)$;
we will call these \emph{peak flux} and \emph{peak density}, respectively.

Recall that under the step initial condition,
all the fluid's mass is concentrated in the negative half-line at time $t=0$.
The key characteristic of the system is the amount of mass 
transported across the origin to the positive half-line by the time $t=1$,
which turns out to be precisely the peak flux:
\begin{gather*}
\int_{0}^{+\infty} y(x) dx =
\int_{0}^{f'(0)} (f')^{-1}(x) dx =
\int_{0}^{\tilde y} f'(y) dy = 
f(\tilde y) - f(0) = f(\tilde y) = \tilde f
\end{gather*}

We will call the function $1-f = \bar f$ and the value $1 - \tilde f = \bar{\tilde f}$
respectively the \emph{flux complement} function and the \emph{peak flux complement}.
A close relationship between the peak flux complement and the constant $\gamma$ of the Chv\'atal--Sankoff problem
will be exposed in the rest of this section.

\paragraph{Network model limit}

For a network evolution model, density $y$ in the above equations 
is the limiting marginal probability of a site to contain a particle (as opposed to a hole).
The flux for a model $X$ is determined as the (unconditional) probability 
that a particle and a hole are exchanged by a comparator within the cell.
This probability, as well as its complement, have a straightforward expression
in terms of marginal site probabilities:
{\small \begin{gather}
\label{eq-flux}
f^X \eqdef \cell[1]{1:0}{O:I} = \cell{1:}{0:}  = \cell{.:}{0:} - \cell{0:}{O:} = \cell{:}{0:} - \cell{0:}{:}\qquad
\bar f^X \eqdef 1 - (\cell{:}{0:} - \cell{0:}{:}) = \cell{:}{1:} + \cell{0:}{:}
\end{gather}}%
For a model evolving vertically or horizontally,
every cell is accounted for in the above expession for the flux in a given time step.
For a model evolving diagonally,
one half of the cells is accounted for in the first half-step of a time step,
and the other half of the cells in the second half-step.

For a model that has mirror symmetry of cell type probabilities about the main diagonal
(such as model $\mathit{CS}$ and all the others considered in this paper),
and that evolves diagonally from the (skew-symmetric) step initial condition, 
the site probabilities will be skew-symmetric about the main diagonal:
particle probability at a site on one side of the main diagonal 
must be equal to the hole probability at the symmetrically opposite site.
By symmetry, the peak density for such a model in the scaling limit is $\tilde y = \frac{1}{2}$,
realised in a small neighbourhood of the main diagonal.

From now on, we will consider the model's state 
in an infinitesimally small neighbourhood of the scaling limit point $t=1$, $x=0$ on the main diagonal.
At that point, both the model's peak flux and peak density are realised,
so we will write simply $y$ for $\tilde y$ and $f^X$ for $\tilde f^X$.
The peak density $y$ is composed from particle probabilities at horizontal and vertical sites,
or, symmetrically, particle and hole probabilities at just the horizontal, or just the vertical sites:
$y = u + \U = \zhv{1:} + \zhv{:1} = \zhv{1:} + \zhv{0:} = \zhv{:0} + \zhv{:1} = \frac{1}{2}$.
The evolution of the model in such a small neighbourhood
can be considered to be in a stationary state;
we will use this stationarity to derive the joint distribution for site probabilities of our models.

\paragraph{A limit for model $\mathit{CS}$}

In general, finding an explicit flux function for a particle evolution model may be difficult,
and even the convergence to a scaling limit is not guaranteed.
Fortunately, the existence of a continuous scaling limit for model $\mathit{CS}$ 
follows directly from \Cref{p-cs-shape}.
Indeed, the model's convergence at a point on the main diagonal is equivalent 
to the convergence of scaled LCS length for a pair of equally long uniformly random binary strings,
i.e.\ to the existence of constant $\gamma$.
As mentioned in the Introduction, this was established already 
by Chv\'{a}tal and Sankoff \cite{Chvatal_Sankoff:75} (see also \cite[Chapter 1]{Steele:97}).
In much the same way, the model's convergence at any other point is equivalent
to the convergence of scaled LCS length for a pair of random binary strings
with a given limiting ratio of their lengths,
which can be established by a slight modification of the same proof.

The Chv\'atal--Sankoff problem can now be reformulated 
as finding the peak flux complement $\gamma = {\bar f}^{\mathit{CS}}$ for model $\mathit{CS}$.

\section{Model $B$}
\label{s-b}

In keeping with the traditional terminology,
let us define \emph{model $B$} (the Bernoulli model) as the network evolution model,
where a cell is assigned type ``mismatch'' with a fixed probability $\p \eqdef \zvh[1]{:}$,
called the model's \emph{(jump) rate}, independently of any site values or types of any other cells
(this initial definition will be generalised later).
Intuitively, every cell tosses an independent biased coin $p$ to determine its type.

Model $B$ has been applied to the study of the Chv\'atal--Sankoff problem by 
Boutet de Monvel \cite{Boutet:99}, Majumdar and Nechaev \cite{Majumdar_Nechaev:05}, 
Priezzhev and Sch\"{u}tz \cite{Priezzhev_Schuetz:08}, Bukh and Cox \cite{Bukh_Cox:22}.
It is closely related to a classical particle model known as
the \emph{totally asymmetric simple exclusion process (TASEP)}.
The TASEP consists of an of array of sites, each occupied by a particles or a hole.
It evolves by a particle jumping at a random time into a hole on its right;
symmetrically, the hole ``jumps'' to its left to the site previously occupied by the particle.
Updates may occur in continuous time 
(classical TASEP, which we do not consider any further) or in discrete time (DT-TASEP).
Within a time step of DT-TASEP, the update policy may be parallel 
(the process also known as multi-corner growth of a Young diagram,
which we do not consider any further),
forward-sequential, backward-sequential, or sublattice-parallel.
The latter three update policies essentially only differ by a change of coordinates,
and correspond to model $B$ evolving vertically, horizontally or diagonally, respectively.
An analysis of DT-TASEP with different update policies
has been given by Rajewsky et al.\ \cite{Rajewsky_:98} and by Martin and Schmidt \cite{Martin_Schmidt:11}.
Model $B$ and DT-TASEP can be considered as a special case of the six-vertex model
analysed by Borodin et al.\ \cite{Borodin_:16},
with weights assigned according to measure $\mathcal{P}(p,0)$ defined therein.

Model $B$ and other network evolution models presented in this paper
can also be considered as special cases of stochastic cellular automata 
(see e.g.\ \cite{Mairesse_Marcovici:14,Casse:16}).
However, the simplifying ``well-mixing'' assumptions, that are usually made in that context, 
do not hold for our models.

\paragraph{Cell type probabilities}
We note that a cell's type only affects the model's behaviour when its entry pair is $\zvh{1:0}$,
distinguishing the events $\cell[0]{1:0}{I:O}$ and $\cell[1]{1:0}{O:I}$.
For any other entry pairs,
the cell's exit values are forced by the entry values and are independent of the cell's type:
the corresponding events are $\cell{0:0}{O:O}$, $\cell{0:1}{O:I}$, $\cell{1:1}{I:I}$.
In these cases, the cell's type probability $\zvh[1]{:}$
can be set differently from $p$, without affecting the model's behaviour.
Therefore, we can generalise the definition of model $B$ 
by introducing a formal dependency of a cell's type on its entry pair,
while making sure that the model's new definition 
is still invariant with respect to duality of configurations.

\begin{definition}
We say that a cell's type \emph{depends exclusively} on a set of sites' values in a given half-step,
if, conditioned on this set, it is conditionally independent of any other site values in the same half-step.
\end{definition}

We define $4 = 1 \cdot 2 + 2$ (one dual pair and two self-dual singletons) 
conditional probabilities for a cell's type,
specifying its exclusive dependence on the entry site pair:
{\small \begin{gather*}
p_0 \eqdef \zvh[1]{o:o}\eqr
p_3 \eqdef \zvh[1]{i:i}\qquad
p_1 \eqdef \zvh[1]{o:i}\qquad
p_2 \eqdef \zvh[1]{i:o}
\end{gather*}}%
The subscripts correspond to the entry pair values
being read as a two-digit binary number, bottom-left to top-right: $p_0 = p_{\circ\circ}$, etc.
Intuitively, a cell now has four biased coins $p_0$, $p_1$, $p_2$, $p_3$, 
including a dual pair $p_0 \eqr p_3$.
The cell reads its entry pair (as a binary number), and then tosses the corresponding coin to determine its type;
the combination of the cell's entry pair and its chosen type then determines the cell's exit pair.

Conditional probability $p_2$ corresponds to the rate $p$ in the original definition of model $B$,
and determines solely the model's behaviour (in particular, its flux).
We will therefore reserve the term \emph{rate} for $p_2$,
whereas the remaining conditional probabilities $p_0 \eqr p_3$, $p_1$ will be called \emph{pseudo-rates}.
These pseudo-rates do not affect the behaviour of the model,
and therefore can temporarily be left unconstrained.
This leaves us the freedom to set them later, 
in an attempt to fit model $B$ to the constraints of model $\mathit{CS}$.

\paragraph{Alternating sequences}

Our models, including model $B$, will have time-invariant distributions satisfying the following natural property.
\begin{definition}
\label{def-alt}
An \emph{alternating sequence} is a doubly-infinite sequence of (generally dependent)
particle-hole random variables $(\xi_i)$, $i \in \Int$, that is invariant with respect to
\begin{itemize}
\item a shift by $1$, mapping $i \mapsto i+1$
\item a reversal about $\frac{1}{2}$, mapping $i \mapsto -i+1$
\end{itemize}
both of these with simultaneous exchange between particles and holes.
\end{definition}
Note that both a shift and a reversal of the given type flip the parity of indices.
\Cref{def-alt} implies that an alternating sequence is also invariant 
with respect to arbitrary shifts and reversals,
where holes and particles are exchanged if and only if the parity of indices is flipped.

We consider alternating sequences of site values in a given half-step of diagonal network evolution,
identifying arbitrarily the even (respectively, odd) indices of the sequence 
with the horizontal (respectively, vertical) edge sites.
Annotated equality $A \eqr B$, when applied to such sequences, will stand for 
``$A$ and $B$ are obtained from each other by a parity-exchanging reversal 
with a simultaneous exchange between particles and holes'';
this is consistent with the previous usage of this notation to express duality of configurations.
Furthermore, annotated equality $A \eqs B$ will have similar meaning, but with a reversal replaced by a shift.
Notation $A \eqsr B$ will be used when $A \eqs B$ and $A \eqr B$ are both applicable.

We denote the marginal site probabilities by
{\small \begin{gather}
\label{eq-u}
\u \eqdef  \zvh{0:} \eqsr \zvh{:1} \qquad
\U \eqjda \zvh{1:} \eqsr \zvh{:0}
\end{gather}}%
Substituting \eqref{eq-u} into \eqref{eq-flux}, 
we obtain a simple expression for the peak flux complement of our models.
\begin{proposition}
Let $X$ be a network evolution model in a stationary state, 
where the time-invariant distribution of site values is given by an alternating sequence.
Then the peak flux complement is
{\small \begin{gather*}
{\bar f}^X = \cell{:}{1:} + \cell{0:}{:} = \u + \u = 2\u
\end{gather*}}%
\end{proposition}

In the rest of this section and the next, we designate $u$, $p_0$, $p_1$, $p_2$ as the \emph{main variables};
our goal is to connect them by a system of polynomial equations with integer coefficients.
In principle, this could be done directly in terms of the main variables alone;
however, for convenience, we will be introducing some \emph{auxiliary variables}.
Every auxiliary variable will have a separate equation 
expressing it in terms of previously introduced variables;
thus, auxiliary variables will not add any degrees of freedom to the system,
and could easily be eliminated from it, at the expense of making the equations more cumbersome.

\paragraph{AB sequences}
We first consider the most basic special case of an alternating sequence.
\begin{definition}
\label{def-ab}
An alternating sequence $(\xi_i)$, $i \in \Int$, is an \emph{AB (alternating Bernoulli) sequence},
if all its elements are mutually independent.
\end{definition}
In particular, an AB sequence of site values in a given half-step of a network evolution model
is a product measure with marginal site probabilities \eqref{eq-u}.

\paragraph{Time invariance}
Consider the evolution of model $B$ on an AB sequence in a stationary state.
The model's rate and the site densities of the sequence are connected by the \emph{time-invariance equation}:
{\small \begin{gather}
\label{eq-BI}
u u = \zhv{1:0} = \cell[0]{1:0}{I:O} = \U\U \P_2
\end{gather}}%
We recall a well-known result on the time-invariant distribution for the diagonal evolution of model $B$
(see e.g.\ Rajewsky et al.\ \cite{Rajewsky_:98}, Martin and Schmidt \cite{Martin_Schmidt:11}).
\begin{theorem}
\label{th-b-inv}
An AB sequence with parameter $u$ determined by \eqref{eq-BI}
is a time-invariant distribution for model $B$ with a given rate $p_2$.
\end{theorem}
\begin{proof} 
It is sufficient to show that the AB property is preserved in a single half-step of the evolution of model $B$.
The independence between site values $a$, $b$
at the end of the half-step in a configuration $\zvh{a:b}$ is obvious,
since these values are obtained in different cells;
independence in a configuration $\zhv{a:b}$ is established by \eqref{eq-BI}.
In the equation \eqref{eq-BI},
the left-hand side expresses the AB property in a configuration at the half-step's end,
while the right-hand side relies on the same property at the half-step's beginning.
\end{proof}

\paragraph{The Arratia--Steele conjecture}
Since the marginal probabilities of cell types in model $\mathit{CS}$ are all equal to $\frac{1}{2}$,
and since these types are $\cellthree$-independent and, more generally, three-wise independent,
at some point it was quite natural to conjecture that all cell dependence
(i.e.\ $\cellfour$-dependence and, more generally, all four-wise and higher-order dependence) could also be ignored,
so that model $\mathit{CS}$ would be equivalent 
to model $B$ with $p = p_2 = \frac{1}{2}$, which we denote $B(1/2)$.
Substituting $p_2 = \frac{1}{2}$ into \eqref{eq-BI}, we obtain a quadratic equation
that gives us the peak site marginal probability and the peak flux complement for model $B(1/2)$,
which can be considered as an approximation for $\gamma$:
{\small \begin{gather*}
u = \sqrt2-1 = 0.414213\ldots \qquad 
\gamma \approx {\bar f}^{B(1/2)} = 2u = 2(\sqrt2-1) = 0.828427\ldots
\end{gather*}}%
The conjecture, attributed to Arratia by Steele \cite{Steele:86}, 
was that the above expression gives the exact value of $\gamma$.
This conjecture was disproved by the upper bound $\gamma \leq 0.826280$ due to Lueker \cite{Lueker:09}.
 
In the remainder of this paper, we will be making repairs to the Arratia--Steele conjecture
by weakening the claimed type of model equivalence (local instead of global equivalence),
and by replacing
model $B$ by a network evolution model from a more general class.

\section{Local fitting of model $B$ to model $\mathit{CS}$}
\label{s-b-fit}

The Arratia--Steele conjecture makes an unsuccessul attempt 
to fit model $B(1/2)$ to model $\mathit{CS}$.
The next natural step is to replace model $B(1/2)$ by model $B$ with a general rate and pseudo-rates.
In doing so, we can set a rate $p_2 > \frac{1}{2}$ 
for a better fit to the higher peak rate of model $\mathit{CS}$.
We need to compensate for that by lowering (at least one of) the pseudo-rates $p_0 \eqr p_3$, $p_1$,
so that the marginal cell type probability remains at $\frac{1}{2}$,
and the $\cellthree$-independence of cell types is maintained.
Crucially, we do not need to require that the models agree across the whole network:
since we are only interested in obtaining the peak flux complement,
we only need to achieve the models' agreement in a small neighbourhood of the main diagonal
(recall that the peak flux is precisely the flux across the main diagonal);
we call this \emph{local fitting} of the models.
This attempt to obtain a local fit for model $B$ to model $\mathit{CS}$ 
will eventually turn out to also be unsuccessful, 
but somewhat less so than the Arratia--Steele conjecture: 
it will give us a better approximation for $\gamma$,
and is designed to provide a stepping stone to future, still more general models.

For convenience, we introduce a pair of auxiliary variables 
for cell type probabilities conditioned on a single entry site,
where the sites in a given half-step are known to form an AB sequence:
{\small \begin{gather*}
\q_0 \eqdef \zvh[1]{o:} = \zvh[1]{o:0} + \zvh[1]{o:1} = \U p_0 + \u p_1 \eqr
\q_{\bar 1} \eqjda \zvh[1]{:i}\\
\q_1 \eqdef \zvh[1]{i:} = \zvh[1]{i:0} + \zvh[1]{i:1} = \U p_2 + \u p_3 \eqr
\q_{\bar 0} \eqjda \zvh[1]{:o}
\end{gather*}}%

\paragraph{Reverse cell type probabilities}
Forgetting temporarily about cell types, 
the evolution of site values in model $B$ can be ``turned back in time''
by considering a natural reverse process,
where site values in half-step $t$ are conditioned on site values (without cell types) in half-step $t+1$.
Both the forward and the reverse processes on site values can be described symmetrically as
{\small \begin{gather*}
1 = \cell{o:o}{O:O} = \cell{O:O}{o:o} = \cell{i:i}{I:I} = \cell{I:I}{i:i} = 
    \cell{o:i}{O:I} = \cell{I:O}{i:o}\qquad
\p_2 = \cell{i:o}{0:1} = \cell{1:0}{o:i}\qquad
\P_2 = \cell{i:o}{1:0} = \cell{0:1}{o:i} 
\end{gather*}}%
Although it is not required for our results, it is remarkable, and not difficult to check via $\eqref{eq-BI}$,
that in the stationary state, it is impossible to distinguish probabilistically 
whether a given configuration of site values has been obtained by the forward or by the reverse process.

Reintroducing cell types breaks the symmetry between the forward and the reverse processes:
a cell's type determines its operation in the forward process only.
In the reverse process, a cell's type depends exclusively on its pair of exit sites
(here, the terminology ``exit sites'' is still relative to the forward process).
We denote the $4 = 1 \cdot 2 + 2$ resulting reverse conditional probabilities by auxiliary variables
{\small \begin{gather*}
\r_{ab} \eqdef \zhv[1]{\textcond{a}:\textcond{b}} \eqr
\r_{\bar b \bar a} \eqjda \zhv[1]{\textcond{\bar b}:\textcond{\bar a}}
\end{gather*}}%
which are determined by the model's parameters via the equations
{\small \begin{gather}
\R_0 = \zhv[0]{o:o} = \cell[0]{o:o}{O:O} = \P_0\eqr
\R_3 = \zhv[0]{i:i} = \cell[0]{i:i}{I:I} = \P_3 \qquad
\R_1\U\U = \zhv[0]{0:1} = \cell[0]{0:1}{O:I} = u u \P_1 \qquad
\R_2 = \zhv[0]{i:o} = 1
\end{gather}}%

\paragraph{Total probability}
An individual cell in model $\mathit{CS}$ 
takes its types $\zvh[0]{:}$ and $\zvh[1]{:}$ equiprobably.
Therefore, in a local fit of the models,
the site probabilities, rate and pseudo-rates of model $B$ must satisfy the total probability equation,
where the pseudo-rates $p_0$, $p_1$ fulfill their purpose of balancing out the bias in the rate $p_2$:
{\small \begin{gather}
\label{eq-BT}
\U\U p_2 + 2\u\U p_0 + u u p_1 = 
\zvh[1]{1:0} + 2\:\zvh[1]{0:0} + \zvh[1]{0:1} = 
\zvh[1]{:} = \textstyle\frac{1}{2}
\end{gather}}%
Here, we have collected an equiprobable dual pair of terms $\zvh[1]{0:0} \eqr \zvh[1]{1:1}$ 
into a single term $2\:\zvh[1]{0:0}$.
The remaining terms are self-dual singletons.
We shall use the same shortcut subsequently without special notice;
the collected mutually dual terms can always be distinguished by the leading coefficient 2.

\paragraph{Linking the models}
We now attempt to link models $B$ and $\mathit{CS}$.
Our goal is to assign the rate and pseudo-rates for model $B$ so that the site values and cell types 
at any given half-step would be probabilistically indistinguishable in both models.

In model $B$, a cell's type in half-step $t$ depends exclusively on its entry site pair.
Each site of this pair is an exit site for one of a pair 
of diagonally adjacent cells in half-step $t-1$, 
and depends exclusively on those cells' entry site pairs;
we thus have an exclusive dependence of a cell's type in half-step $t$
on a quadruple (two disjoint pairs) of adjacent sites in half-step $t-1$.
Each site of this quadruple, in its turn, is an exit site for one of a triple
of diagonally adjacent cells in half-step $t-2$, 
and depends exclusively on those cells' entry site pairs;
we thus have an exclusive dependence of a cell's type in half-step $t$
on a sextuple (three disjoint pairs) of adjacent sites in half-step $t-2$.
On the other hand, in model $\mathit{CS}$, a cell's type in half-step $t$
is determined uniquely by just the types of three preceding cells,
two in half-step $t-1$ and one in half-step $t-2$,
forming a $\cellthree$-shape between themselves, and a $\cellfour$-shape together with the current cell.
Therefore, in order to relate models $B$ and $\mathit{CS}$,
a system of polynomial equations can be obtained by listing exhaustively 
all possible configurations of the relevant site values and cell types
over three half-steps of both models' evolution.
The use of duality and of the reverse cell type probabilities 
will provide substantial shortcuts for such an exhaustive enumeration,
helping us to avoid listing dozens of configurations explicitly.

There are three linking equations: one for the rate $p_2$, 
and two (considering duality) for the pseudo-rates $p_0 \eqr p_3$, $p_1$.
The left-hand side of every equation represents 
a single half-step configuration for a given rate or pseudo-rate.
The right-hand side enumerates exhaustively each of the three half-step configurations of model $\mathit{CS}$
that result in the configuration in the left-hand side with nonzero probability.
Due to $\cellthree$-independence of cell types, the first two of these steps
are probabilistically indistinguishable from ones of model $B$,
and their probabilities are assigned accordingly in the equations;
by using the reverse probabilities, we avoid an explicit enumeration of the site values in the first half-step,
while we do enumerate cell types.
In the third half-step, the cell type is determined uniquely 
from the $\cellfour$-configuration of model $\mathit{CS}$;
the requirement that this half-step must also be probabilistically indistinguishable from one of model $B$
provides the desired equation.

In particular, the linking equation for the rate $p_2$ is as follows:
\begin{subequations}
\label{eq-BS}
\def\phan{\phantom{\u\U \p_0}}
{\small \begin{gather}
\mathrlap{\U\U \p_2}\phan = \zvh[1]{1:0} = 
\Bigpa{
\animal[x]{1:1:1:I}{:::}{:1:0:}{:I:O:} +
2\:\animal{0:1:0:I}{:::}{:1:0:}{:I:O:} +
\animal[x]{1:0:0:I}{:::}{:1:0:}{:I:O:}} +
2\:\Bigpa{
\animal{1:1:1:I}{:::}{1:0:0:}{:I:O:} +
\animal{0:1:0:I}{:::}{1:0:0:}{:I:O:}} + 
\animal{1:1:1:I}{:::}{1:0:1:0}{:I:O:} \\
\notag\phan = 
2 \u\u \R_2\q_0\Q_0 + 
2\U\U\u (\r_0\p_2\q_0 + \R_0\p_2\Q_0) +
\U\U\U\U \r_1\p_2\p_2
\end{gather}}%
The terms representing impossible events have been crossed out and dropped from the equation;
from now on, such terms will be omitted without special notice.

The remaining linking equations are as follows:
{\small \begin{gather}
\mathrlap{\u\U \p_0}\phan = \zvh[1]{0:0} = 
\Bigpa{
\animal{0:0:1:I}{:::}{:0:0:}{:O:O:} +
\animal{1:0:0:I}{:::}{:0:0:}{:O:O:}} +
\Bigpa{
\animal{1:1:1:I}{:::}{0:0:0:}{:O:O:} +
\animal{0:1:0:I}{:::}{0:0:0:}{:O:O:}} +
\animal{1:1:1:I}{:::}{0:0:1:0}{:O:O:} +
\animal{0:0:1:I}{:::}{:0:1:0}{:O:O:}\\
\notag\phan = 
\U\u (\R_0\Q_1\q_0 + \r_0\Q_1\Q_0)+
\u\U\u (\r_0\p_0\q_0 + \R_0\p_0\Q_0) + 
\u\U\U\U \r_1\p_0\p_2 +
\U\U\U \R_1\Q_1\p_2\\
\mathrlap{\u\u \p_1}\phan = \zvh[1]{0:1} =
\animal{1:1:1:I}{:::}{0:0:1:1}{:O:I:} +
2\:\animal{0:1:0:I}{:::}{0:0:1:}{:O:I:} +
\animal{1:0:0:I}{:::}{:0:1:}{:O:I:} =
\u\U\U\u \r_1\p_0\p_3 +
2\u\U\U \R_1\p_0\Q_1 + 
\U\U \r_1\Q_1\Q_1
\end{gather}}%
\end{subequations}

It is important to note that the connection between the two models 
expressed by equations \eqref{eq-BS} is incomplete:
while the equations relate the rate and the pseudo-rates of model $B$ to model $\mathit{CS}$,
they do not guarantee the preservation of the AB property on the site values.
In particular, the equations' left-hand sides 
express site independence within a configuration of the form $\zvh{a:b}$, 
but none of the equations implies site independence within a configuration of the form $\zhv{a:b}$.
Thus, there is no guarantee that our goal of probabilistic indistinguishability 
between the two models has been achieved: 
in fact, it has not, and in general model $B$ turns out to be insufficient for a perfect fit.

\paragraph{Solving the equations}
The resulting system has four main variables $u$, $p_0$, $p_1$, $p_2$, 
involved in five main equations: one time-invariance equation \eqref{eq-BI},
one total probability equation \eqref{eq-BT}, 
and three linking equations \eqref{eq-BS}.
There are also some auxiliary variables, each of which is introduced via its own separate equation.
Thus, the system is overdetermined by one equation.
However, it is still consistent, since the total probability equation \eqref{eq-BT}
is a consequence of the $\cellthree$-independence of cell types,
which is implied by the time-invariance and the linking equations.
While the total probability equation is formally redundant,
we keep it in the system for its symmetry and, more importantly, 
as an aid to computer algebra software in solving the system.

Since all the variables in the system represent probabilities,
we are only interested in real solutions between $0$ and $1$;
we call such solutions \emph{admissible}.
The system has a unique admissible solution;
we denote by $B(\mathit{opt})$ model $B$ with the specific set of parameters provided by this solution.

The system's admissible solution can be obtained analytically using computer algebra software.
In particular, Mathematica returns it instantly, expressed in exact radicals 
(for this, function \texttt{Solve} needs to be used with option \texttt{Quartics -> True}).
As a result, we obtain the site marginal probability 
and an estimate for $\gamma$ via the peak flux complement as
{\small \begin{gather*}
\textstyle 
u = \sqrt{\frac{7}{3}}-\sqrt{\frac{23-5 \sqrt{21}}{6}}-1 = 0.407025\ldots \qquad
\gamma \approx {\bar f}^{B(\mathit{opt})} = 2 u = 0.814050\ldots
\end{gather*}}%
and the model's rate and pseudo-rates as
{\small \begin{gather*}
p_0 = p_3 = -\tfrac{8}{3} + \tfrac{49}{6} u - uu - \tfrac{1}{2} uuu = 0.457987\ldots\\
p_1 = \tfrac{29}{2} - 51 u + \tfrac{75}{2} uu + 9 uuu = 0.561206\ldots \\
p_2 = -\tfrac{2}{3} + \tfrac{34}{3} u - 19 uu - 4 uuu = 0.528838\ldots
\end{gather*}}%

We have thus attempted to obtain a local fit of model $B$ to model $CS$,
expressing the various constraints of the latter by polynomial equations with integer coefficients,
and obtaining the unique admissible solution of the resulting equation system as model $B(\mathit{opt})$.
However, this fit is not perfect, since the AB sequence property is not preserved;
enforcing its preservation by introducing additional equations 
would make the system truly overdetermined and inconsistent.
We conclude that not only model $B(1/2)$ of the Arratia--Steele conjecture,
but even the more general model $B$
is still too rigid to provide a perfect local fit to model $CS$.

\section{Conclusion}

In this paper, we have linked the Chv\'atal--Sankoff problem
to the parameters of a certain stochastic particle process (model $B$),
using existing results on the combinatorial structure of the LCS problem
and the theory of continuous scaling limits for discrete particle processes.
Model $B$ generalises the Bernoulli model that has been used to obtain 
previous estimates for the Chv\'atal--Sankoff constant $\gamma$.
We have obtained a system of polynomial equations with integer coefficients
that determines the parameters for this process.
The system is solvable in exact radicals,
and provides an estimate for the Chv\'atal--Sankoff constant
that is substantially better than the one obtained from the original Bernoulli model.
At the same time, our technique differs from the Monte Carlo simulations
used previously to obtain numerical approximations for $\gamma$,
and complements the insights obtained from those simulations. 

We anticipate that still better estimates can be obtained 
with even more general stochastic particle models with more parameters.

\bibliography{research}

\end{document}